\documentclass[a4paper,11pt]{article}
\usepackage{amsmath,amssymb,amsthm,graphics,latexsym,amsfonts}
\usepackage{fancyhdr}
\usepackage[nooneline,flushleft]{caption2}
\usepackage{dsfont}
\usepackage{mathtools}
\usepackage{subeqnarray}
\usepackage{bbding,booktabs}
\usepackage{graphicx,color}
\usepackage[a4paper,top=1in, bottom=20ex, left=25mm,
right=25mm,headheight=25pt,headsep=6pt]{geometry}

\newtheorem{Lemma}{\quad Lemma}[section]

\numberwithin{equation}{section}

\title{\Large\bf
A fully discrete local discontinuous Galerkin method with the generalized
numerical flux to solve the tempered fractional reaction-diffusion
equation }
\author{{ Leilei Wei$^{a}$, Yinnian He$^{b}$
\footnote{Corresponding author. \newline $~~~~~~$E-mail addresses:
$^{a}$leileiwei@haut.edu.cn, $^{a}$leileiwei09@gmail.com, $^{b}$heyn@mail.xjtu.edu.cn.
}
}\\
\footnotesize  \emph{a. College of Science, Henan University of Technology,
Zhengzhou 450001, Henan Province, PR China}
\\
\footnotesize  \emph{b. School of Mathematics and Statistics, Xi’an Jiaotong University, Xi’an 710049, PR China}
}
\date{}

\newtheorem{lem}{Lemma}[section]
\newtheorem{thm}[lem]{Theorem}

\begin{document}
\maketitle
{\small
\noindent{\bfseries Abstract:}
The tempered fractional diffusion equation could be recognized as the generalization of
the classic fractional diffusion equation that the truncation effects are included in the bounded domains.
This paper focuses on designing the high order fully discrete local
discontinuous Galerkin (LDG) method based on the generalized alternating numerical fluxes
for the tempered fractional diffusion equation.
From a practical point of view,
the generalized alternating numerical flux which is different from the purely alternating numerical flux
has a broader range of applications.
We first design an efficient finite difference scheme to approximate the tempered fractional derivatives
and then a fully discrete LDG method for the tempered fractional diffusion equation.
We prove that the scheme is unconditionally stable and convergent with the order
$O(h^{k+1}+\tau^{2-\alpha})$,
where $h, \tau$ and $k$ are the step size in space, time and the degree of piecewise polynomials, respectively.
Finally
numerical experimets are performed to show the effectiveness and testify
the accuracy of the method.

}

\section {Introduction}

Anomalous diffusion models which better describe transport processes in complex heterogeneous systems
are the fluid limit of a time random walk with a probability distribution function
for the displacements and the waiting times.
In fact, from a view of practice, upper bounds are existed on the waiting times between the displacements that a particle could do or on these displacements.
Thus one should consider these upper bounds in the model.
In order to recover finite moments,  truncating waiting time probability distribution function and the displacements is
a feasible method.
Mantegna and Stanley \cite{mss94}
and Koponen \cite{kop95}  applied truncated L\'{e}vy processes
to get rid of large displacements.
Rosi\'{n}ski \cite{roz07}
replaced the sharp cutoff by a smooth exponential damping of the tails of the probability distribution function.
In the fluid limits, exponentially tempered L\'{e}vy processes result in a tempered fractional diffusion equation \cite{bm10}.
If the waiting times be tempered, a fractional diffusion equation with a tempered time derivative will be modeled \cite{mmz08}.

In recent years, many numerical methods are presented to solve
fractional subdiffusion and superdiffusion equations, for example,
finite difference methods
\cite{fd2,fd11,fd5,lm18,fd6,fd7,fd8,fd9,fd10},
finite element methods\cite{fe12,fe13,fe14,fe15,fe16,fe7},
spectral methods\cite{sp1,sp2,sp3},
discontinuous Gakerkin methods \cite{wei1}.
Other numerical methods such as homotopy perturbation method
and the variational method also works very efficiently,
for details the readers can refer to \cite{liao1,ot1}.
However, compared with a lot of papers on numerical methods of
fractional partial
differential equations, the literature about the tempered fractional differential equations is limited.
Li and Deng \cite{li14} proposed some high-order schemes based
on the weighted and shifted Gr\"{u}nwald difference
operators to solve the space tempered fractional diffusion equations.
In \cite{ydw17},  Yu et al. analyzed the third order difference schemes
for the space tempered fractional diffusion equations.
In \cite{bm10}, Baeumera and Meerschaert developed
a finite difference scheme to solve the tempered fractional diffusion equation,
and discussed the stability and convergence of the method.
Cartea and delCastillo-Negrete \cite{cdd07} considered
a finite difference method for the tempered fractional Black-Merton-Scholes
equation.
Hanert and Piret \cite{hpsiam14} developed a pseudo-spectral method
based on a Chebyshev expansion in space and time to
discretize the space-time fractional diffusion equation
with exponential tempering in both space and time,
and proved that the proposed
scheme yields an exponential convergence when the
solution is smooth.
Zhang et al. \cite{zhang16} discussed a high-order finite difference scheme
for the tempered fractional Black-Scholes equation.
Hao et al. \cite{hao17} discussed a second-order difference scheme
for the time tempered fractional diffusion equation,
and analyzed its stability and convergence.

In order to broaden the applicable range of tempered fractional diffusion models,
it is meaningful and challenging
to construct high-order numerical schemes for the model equation.
The discontinuous Galerkin (DG) methods is naturally formulated
for any order of accuracy in each element,
and is flexibility and efficiency in terms of mesh and shape functions
\cite{zhang14}.
In this paper,
we will present a fully discrete local discontinuous Galerkin
(LDG) method based on generalized alternating numerical fluxes
to solve the tempered fractional diffusion equation
\begin{equation}\label{question}
_0^{\rm C}D_t^{\alpha,\gamma}u(x,t)+\rho u(x,t)-
\frac{\partial^{2}u(x,t)}{\partial x^{2}}=f(x,t),
\quad (x,t)\in (a,b)\times (0,T],
\end{equation}
with the initial solution $u(x,0)=u_0(x)$,
where $f$ and $u_0$ are given smooth functions, and
$\rho>0$ is the constant reaction rate.
Here
\begin{equation}
_0^{\rm C}\mathcal{D}_t^{\alpha,\gamma}u(x,t)
=
e^{-\gamma t}~~_{0}^{C}D_t^{\alpha}\left[e^{\gamma t}u(x,t)\right]
=
\frac{e^{-\gamma t}}{\Gamma(1-\alpha)}
\int_0^{t}\frac{\partial}{\partial s}
\left[e^{\gamma s}u(x,s)\right]\frac{ds}{(t-s)^{\alpha}},
\end{equation}
is the tempered fractional derivative of order $0<\alpha<1$,
with $\Gamma(\cdot)$ being the Gamma function.
In this paper we do not pay attention to boundary condition,
hence the solution is considered to be either periodic.

The outline of the paper is as follows. We first introduce
some basic notations and preliminaries which will be used later.
Then in Sec. 3
a fully discrete LDG method for the tempered fractional equation
\eqref{question}, and also
discuss its stability and convergence.
Numerical examples are provided
to show the accuracy and capability of the scheme in Sect. 4.
Some concluding remarks are given in the final section.

%
%
%
%
%

\section {Fully-discrete LDG scheme}

As the usual treatment in LDG method,
we rewrite equation (\ref{question}) as the equivalent first-order system:
\begin{equation}\label{eq3}
p=u_x,\quad
_0^CD_t^{\alpha,\gamma}u(x,t)+\rho u(x,t)-p_x=f.
\end{equation}
Let $M$ be a positive integer, and denote
$\tau=T/M$ be the time step and $t_n=n\tau$ be mesh point,
with $n=0,1,\ldots, M$.
Namely, we would like to seek the approximation solutions
$u_h^n$ and $p_h^n$ in the discontinuous finite element space.

Firstly consider the discretization of tempered fractional derivative
$ _0^CD_t^{\alpha,\gamma}u(x,t)$.
At any time level $t_n$, it is approximated as follows
\begin{equation}\label{tf1}
\aligned
_0^CD_t^{\alpha,\gamma}u(x,t_n)
=&\;
\frac{e^{-\gamma t_n}}{\Gamma(1-\alpha)}\sum\limits_{i=0}^{n-1}
\int_{t_i}^{t_{i+1}}
\frac{\partial }{\partial s}[e^{\gamma s}u(x,s)]\frac{ds}{(t_n-s)^{\alpha}}
=\Phi^n(x)+\mu^n(x),
\endaligned
\end{equation}
where
\begin{equation}
\Phi^n(x)
\aligned
=&\;
\frac{e^{-\gamma t_n}}{\Gamma(1-\alpha)}\sum\limits_{i=0}^{n-1}
\int_{t_i}^{t_{i+1}}\frac{e^{\gamma t_{i+1}}u(x,t_{i+1})-e^{\gamma t_{i}}u(x,t_{i})}{\tau}\frac{ds}{(t_n-s)^{\alpha}},
\endaligned
\end{equation}
and $\mu^n(x)$ is the truncation error in time direction.
Similar to the proof in \cite{sp3}, we can know that
\begin{equation}
\|\mu^n(x)\|\leq C\tau^{2-\alpha},
\end{equation}
where the bounding constant $C>0$ depends on $T, \alpha$ and $u$.
Further manipulation yields
\begin{equation}\label{tf1}
\aligned
\Phi^n(x)=&\;
\frac{\tau^{1-\alpha}}{\Gamma(2-\alpha)}\sum\limits_{i=0}^{n-1}b_{n-i-1}
\frac{e^{-\gamma \tau (n-i-1)}u(x,t_{i+1})-e^{-\gamma \tau (n-i)}
u(x,t_{i})}{\tau}
\\
=&\;
\frac{\tau^{1-\alpha}}{\Gamma(2-\alpha)}
\sum\limits_{i=0}^{n-1}b_i
\frac{e^{-i\gamma \tau }u(x,t_{n-i})-e^{-(i+1)\gamma \tau }u(x,t_{n-i-1})}{\tau}
\\
=&\;
\frac{(\Delta
t)^{-\alpha}}{\Gamma(2-\alpha)}(u(x,t_{n})+\sum\limits_{i=1}^{n-1}
(b_i-b_{i-1})e^{-i\gamma \tau }u(x,t_{n-i})-b_{n-1}e^{-n\gamma \tau }u(x,t_{0})),
\endaligned
\end{equation}
where $b_i=(i+1)^{1-\alpha}-i^{1-\alpha}$.

Let
$a=x_{\frac{1}{2}}<x_{\frac{3}{2}}<\cdots<x_{N+\frac{1}{2}}=b$,
where $N$ is an integer.
Define $I_{j}=[x_{j-\frac{1}{2}}, x_{j+\frac{1}{2}}]$
with the cell length $h_j=x_{j+\frac{1}{2}}-x_{j-\frac{1}{2}}$,
for $j=1,\ldots N$,
and denote
$h=\max\limits_{1\leq j\leq N}h_j$.
We assume the partition is quasi-uniform mesh, that is,
there exists a positive constant $\kappa$ independent of $h$ such that
$h_j\geq \kappa h$.
The associated discontinuous Galerkin space $V_h^k$ is defined as
\begin{equation}
V_h^k=\{v:v\in P^k(I_j),x\in I_j,j=1,2,\cdots N\}\nonumber.
\end{equation}
As usual, at each cell interfaces $x_{j+\frac12}$,
we use $v_{j+\frac{1}{2}}^+$ and
$v_{j+\frac{1}{2}}^-$, respectively, to denote the values
from the right cell $I_{j+1}$ and from the left cell $I_j$.
Furthermore, the jump and the weighted average are denoted by
\begin{equation}
[v]_{j+\frac{1}{2}}=u_{j+\frac{1}{2}}^+-u_{j+\frac{1}{2}}^-,
\quad
v_{j-\frac{1}{2}}^{(\delta)}=
\delta v_{j+\frac{1}{2}}^++(1-\delta)v_{j+\frac{1}{2}}^-,
\end{equation}
where $\delta$ is the given weight.

Now we are ready to define the fully-discrete LDG scheme.
The numerical solutions satisfy
\begin{subequations}\label{scheme}
\begin{alignat}{1}
&\;
\left(\rho+\frac{\tau^{-\alpha}}{\Gamma(2-\alpha)}\right)
\int_{\Omega}u_h^nvdx
+\int_{\Omega}p_h^nv_xdx-\sum\limits_{j=1}^{N}
\left((\widehat{p_h^n}v^-)_{j+\frac{1}{2}}-(\widehat{p_h^n}v^+)_{j-\frac{1}{2}}\right)
\nonumber\\
=&\;
\frac{\tau^{-\alpha}}{\Gamma(2-\alpha)}
\left(\sum\limits_{i=1}^{n-1}(b_{i-1}-b_i)e^{-i\gamma \tau}
\int_{\Omega}u_h^{n-i}vdx+
b_{n-1}e^{-n\gamma \tau}\int_{\Omega}u_h^{0}vdx\right)+\int_{\Omega}f^{n}vdx,
\\
&\;
\int_{\Omega}p_h^nwdx+\int_{\Omega}u_h^nw_xdx-\sum\limits_{j=1}^{N}((\widehat{u_h^n}w^-)_{j+\frac{1}{2}}-(\widehat{u_h^n}w^+)_{j-\frac{1}{2}})=0,
\end{alignat}
\end{subequations}
for all test functions $v$ and $w\in V_h^k$.
The hat terms in \eqref{scheme} in the cell boundary terms from
integration by parts are the so-called numerical fluxes.
Instead of using the purely alternating
numerical fluxes as \cite{wei2,xu,xia,wei8}, the novelty of this paper is taking the following generalized alternating
numerical fluxes
\begin{equation}\label{flux1}
\widehat{u_h^n}=(u_h^n)^{(\delta)},
\quad
\widehat{p_h^n}=(p_h^n)^{(1-\delta)},
\end{equation}
with a given parameter $\delta\neq \frac{1}{2}$.
If taking $\delta=0$ or $1$, it will be purely alternating
numerical fluxes.
%

For the convenience of the notations and analysis,
we would like to introduce a compact form of the above scheme.
Adding up two equations in \eqref{scheme}, we have
\begin{equation}\label{scheme1}
\begin{split}
&\;
\left(\rho+\frac{\tau^{-\alpha}}{\Gamma(2-\alpha)}\right)
\int_{\Omega}u_h^nvdx+\int_{\Omega}p_h^nwdx+\mathcal{F}_{\Omega}(u_h^n,p_h^n;w, v)
\\
=&\;
\frac{(\tau^{-\alpha}}{\Gamma(2-\alpha)}
\left(\sum\limits_{i=1}^{n-1}(b_{i-1}-b_i)e^{-i\gamma \tau}\int_{\Omega}u_h^{n-i}vdx
+b_{n-1}e^{-n\gamma \tau}\int_{\Omega}u_h^{0}vdx\right)
+\int_{\Omega}f^{n}vdx,
\end{split}
\end{equation}
where
\begin{equation}
\begin{split}
\mathcal{F}_{\Omega}(u_h^n,p_h^n;w, v)=&\int_{\Omega}u_h^nw_xdx-\sum\limits_{j=1}^{N}(((u_h^n)^{(\delta)}w^-)_{j+\frac{1}{2}}-((u_h^n)^{(\delta)}w^+)_{j-\frac{1}{2}})\\
&+\int_{\Omega}p_h^nv_xdx-\sum\limits_{j=1}^{N}(((p_h^n)^{(1-\delta)}v^-)_{j+\frac{1}{2}}-((p_h^n)^{(1-\delta)}v^+)_{j-\frac{1}{2}}).
\end{split}
\end{equation}

\section{Stability analysis}

For the sake simplifying the notations and without lose of generality, we take $f=0$ in its numerical analysis. For the stability for the scheme (\ref{scheme}), we have the
following result.

\begin{thm}\label{sta}
For periodic or compactly supported boundary conditions,
the fully-discrete LDG scheme (\ref{scheme}) is
unconditionally stable,
and the numerical solution $u_h^n$ satisfies
\begin{equation}
\begin{split}
&\|u_h^{n}\|\leq
\|u_h^0\|,~~n=1,2\cdots, M.
\end{split}
\end{equation}
\end{thm}

\begin{proof}
Taking the test functions $v=u^n_h, w= p^n_h$ in (\ref{scheme1}),
we obtain
\begin{equation}\label{sta2}
\begin{split}
&\;
\left(\rho+\frac{\tau^{-\alpha}}{\Gamma(2-\alpha)}\right)
\|u_h^n\|^2+\|p_h^n\|^2+\mathcal{F}_{\Omega}(u_h^n,p_h^n;p^n_h, u^n_h)
\\
=&\;
\frac{\tau^{-\alpha}}{\Gamma(2-\alpha)}
\left(\sum\limits_{i=1}^{n-1}(b_{i-1}-b_i)e^{-i\gamma \tau}\int_{\Omega}u_h^{n-i}u^n_hdx
+b_{n-1}e^{-n\gamma \tau}\int_{\Omega}u_h^{0}u^n_hdx
\right).
\end{split}
\end{equation}

In each cell $I_{j}=[x_{j-\frac{1}{2}}, x_{j+\frac{1}{2}}]$, we have
\begin{equation}\label{one}
\begin{split}
\mathcal{F}_{I_{j}}(u_h^n,p_h^n;p^n_h, u^n_h)=&\int_{I_{j}}u_h^n(p^n_h)_xdx-((u_h^n)^{(\delta)}(p^n_h)^-)_{j+\frac{1}{2}}+((u_h^n)^{(\delta)}(p^n_h)^+)_{j-\frac{1}{2}}\\
&+\int_{I_j}p_h^n(u^n_h)_xdx-((p_h^n)^{(1-\delta)}(u^n_h)^-)_{j+\frac{1}{2}}+((p_h^n)^{(1-\delta)}(u^n_h)^+)_{j-\frac{1}{2}}\\
=&((u_h^n)^{-}(p^n_h)^-)_{j+\frac{1}{2}}-((u_h^n)^{+}(p^n_h)^+)_{j-\frac{1}{2}} -((u_h^n)^{(\delta)}(p^n_h)^-)_{j+\frac{1}{2}}\\
&+((u_h^n)^{(\delta)}(p^n_h)^+)_{j-\frac{1}{2}}-((p_h^n)^{(1-\delta)}(u^n_h)^-)_{j+\frac{1}{2}}+((p_h^n)^{(1-\delta)}(u^n_h)^+)_{j-\frac{1}{2}}.
\end{split}
\end{equation}

After some detailed analysis, and sum (\ref{one}) from $1$ to $N$ over $j$, we can get the following identity
\begin{equation}\label{st}
\mathcal{F}_{\Omega}(u_h^n,p_h^n;p^n_h, u^n_h)=0.
\end{equation}
Next, we prove Theorem \ref{sta} by mathematical introduction.

Notice the fact that
$e^{-i\gamma \tau}\leq 1$ for any $i\geq0$.
Let
$n=1$ in (\ref{sta2}), we can obtain
\begin{equation*}
\begin{split}
\left(\rho+\frac{\tau^{-\alpha}}{\Gamma(2-\alpha)}\right)
\|u_h^1\|^2+\|p_h^1\|^2
=&\;
\frac{\tau^{-\alpha}}{\Gamma(2-\alpha)}
b_{0}e^{-\gamma \tau}\int_{\Omega}u_h^{0}u^1_hdx
\\
\leq&\;
\frac{\tau^{-\alpha}}{\Gamma(2-\alpha)}
b_{0}\|u_h^{0}\|\|u^1_h\|.
\end{split}
\end{equation*}
Since $b_0=1$, we have
$\|u_h^1\|\leq \|u_h^0\|$.

Suppose that we have proved for the given integer $P$
the following inequalities
\begin{equation*}
\|u_h^m\|\leq \|u_h^0\|, \quad m=1,2,\cdots,P.
\end{equation*}
Letting $n=P+1$ and taking the test functions
$v=u_h^{P+1}$ and $w=p_h^{P+1}$
in \eqref{sta2}, we can obtain
\begin{equation*}
\begin{split}
&\;
\left(\rho+\frac{\tau^{-\alpha}}{\Gamma(2-\alpha)}\right)
\|u_h^{P+1}\|^2
+\|p_h^{P+1}\|^2
\\
\leq&\;
\frac{\tau^{-\alpha}}{\Gamma(2-\alpha)}
\left(\sum\limits_{i=1}^{n-1}
(b_{i-1}-b_i)\|u_h^{P+1-i}\|\|u^{P+1}_h\|
+b_{P}\|u_h^{0}\|\|u^{P+1}_h\|
\right)
\\
\leq&\;
\frac{\tau^{-\alpha}}{\Gamma(2-\alpha)}
\left(\sum\limits_{i=1}^{P}(b_{i-1}-b_i)+b_P\right)
\|u_h^{0}\|\|u_h^{P+1}\|
\\
=&\;
\frac{\tau^{-\alpha}}{\Gamma(2-\alpha)}
b_0\|u_h^{0}\|\|u_h^{P+1}\|.
\end{split}
\end{equation*}
Hence, we can get the following inequality easily
\begin{equation*}
\|u_h^{P+1}\|\leq \|u_h^0\|,
\end{equation*}
which implies the conclusion of this theorem.
\end{proof}

\section{Error estimate}

In this section we present the error estiamte.
To do that, let us firstly recall some important results.

For any periodic function $\omega$ defined on $[a,b]$,
the generalized Gauss-Radau projection \cite{meng16,zhang17},
denoted by $\mathcal{P_\delta\omega}$,
is the unique element in $V_h$.
Let $\omega^e=\mathcal{P_\delta}\omega-\omega$
be the projection error.
When $\delta\neq \frac{1}{2}$, it satisfies
for $j=1,2,\ldots,N$, that
\begin{equation}\label{p1}
\int_{I_j}\omega^evdx=0,
\quad \forall v\in P^{k-1}(I_j),
\quad
\mbox{and  }
(\omega^e)^{(\delta)}_{j+\frac{1}{2}}=0.
\end{equation}
We have the following conclusion \cite{zhang17}.

\begin{Lemma}\label{projection111}
Let $\delta\neq \frac{1}{2}$.
If $\omega\in H^{s+1}[a,b]$, there holds
\begin{equation}
\|\omega^e\|+h^{\frac{1}{2}}\|\omega^e\|_{L^2(\Gamma_h)}
\leq
Ch^{min(k+1,s+1)}\|\omega\|_{s+1},
\end{equation}
where the bounding constant $C>0$ is independent of $h$
and $\omega$. Here $\Gamma_h$
denotes the set of boundary points of all elements $I_j$, and
$$
\|\omega^e\|_{L^2(\Gamma_h)}=
\left(
\frac{1}{2}\sum\limits_{i=1}^{N}
[((\omega^e)^+)^2_{i-\frac{1}{2}}+((\omega^e)^-)^2_{i+\frac{1}{2}}]
\right)^{\frac{1}{2}}.
$$
\end{Lemma}
Also we will use the following conclusion \cite{zzz,fe15}.

\begin{Lemma}\label{tf3}
If $\psi^n\geq 0,n=1,2,\cdots,N, \psi^0=0, \chi>0$, $d_i>0, i=1,2,\dots,l$,
\begin{equation}
\begin{split}
\psi^n
\leq&\sum\limits_{k=1}^{n-1} (b_{k-1}-b_{k})\psi_{n-k}
+\chi,
\end{split}
\end{equation}
then we have
\begin{equation*}
\psi^n\leq C(\tau)^{-\alpha} \chi,
\end{equation*}
where $C$ is a positive constant independent of $h$ and $\tau$.
\end{Lemma}

In the present paper we use the notation $C$ to denote
a positive constant which may have a different value
in each occurrence.
The usual notation of norms in Sobolev spaces will be used.
Denote by $(\cdot,\cdot)_{D}$ the inner product on $L^2(D)$,
with the associated norm by $\|\cdot\|_{D}$.
If $D=\Omega$, we drop $D$.

Now we are ready to present the following estimate.

\begin{thm}\label{errorresult} Let $u(x,t_n)$ be the exact solution of the
problem (\ref{question}), which is sufficiently smooth with bounded
derivatives, Let $u_h^n$ be the numerical solution of the fully
discrete LDG scheme (\ref{scheme}), then there holds the following
error estimates
\begin{equation*}\label{err}
\|u(x,t_n)-u_h^n\|\leq C(h^{k+1}+\tau^{2-\alpha}),
\end{equation*}
where $C$ is a constant depending on $u, T, \alpha$.
\end{thm}

\begin{proof}
Consider the seperation of numerical error in the form
\begin{equation}\label{errnotation}
\begin{split}
&e_u^n=u(x,t_n)-u_h^n=\xi_u^n-\eta_u^n,
\quad
\xi_u^n=\mathcal{P}_{\delta}e_u^n,
\quad
\eta_u^n=\mathcal{P}_{\delta}u(x,t_n)-u(x,t_n),
\\
&e_p^n=p(x,t_n)-p_h^n=\xi_p^n-\eta_p^n,
\quad
\xi_p^n=\mathcal{P}_{1-\delta}e_p^n,
\quad
\eta_p^n=\mathcal{P}_{1-\delta}p(x,t_n)-p(x,t_n).
\end{split}
\end{equation}
Here $\eta_u^n$ and $\eta_p^n$ have been estimated by
Lemma \ref{projection111}.
In what following we are going to estimate
$\xi_u^n$ and $\xi_p^n$.

Since the fluxes \eqref{flux1} are consistent,
we can obtain the following error equation
\begin{equation}\label{err1}
\begin{split}
&\;
\left(\rho+\frac{\tau^{-\alpha}}{\Gamma(2-\alpha)}\right)
\int_{\Omega}e_u^nvdx+\int_{\Omega}\mu^{n}(x)vdx+
\int_{\Omega}e_p^nwdx
+\mathcal{F}_{\Omega}(e_u^n,e_p^n;v, w)
\\
&\;\hspace{1cm}
-\frac{\tau^{-\alpha}}{\Gamma(2-\alpha)}
\left(\sum\limits_{i=1}^{n-1}(b_{i-1}-b_i)e^{-i\gamma \tau}\int_{\Omega}e_u^{n-i}vdx+b_{n-1}e^{-n\gamma \tau}\int_{\Omega}e_u^{0}vdx\right)=0.
\end{split}
\end{equation}
Based on the error decomposition (\ref{errnotation}),
and taking the test functions $v=\xi_u^n$ and
$w=\xi_p^n$ in \eqref{err1}, we have the following
error equations
\begin{equation}\label{err2}
\begin{split}
&\;
\left(\rho+\frac{\tau^{-\alpha}}{\Gamma(2-\alpha)}\right)
\|\xi_u^n\|^2
+\|\xi_p^n\|^2+
+\mathcal{F}_{\Omega}(\xi_u^n,\xi_p^n;\xi_u^n, \xi_p^n)
\\
=&\;
\frac{\tau^{-\alpha}}{\Gamma(2-\alpha)}
\left(\sum\limits_{i=1}^{n-1}(b_{i-1}-b_i)e^{-i\gamma \tau}\int_{\Omega}\xi_u^{n-i}\xi_u^ndx
+b_{n-1}e^{-n\gamma \tau}\int_{\Omega}\xi_u^{0}\xi_u^ndx\right)
\\
&\;
+\left(\rho+\frac{\tau^{-\alpha}}{\Gamma(2-\alpha)}\right)
\int_{\Omega}\eta_u^n\xi_u^ndx
-\int_{\Omega}\mu^{n}(x)\eta_u^ndx
+\int_{\Omega}\eta_p^n\xi_p^ndx
+\mathcal{F}(\eta_u^n,\eta_p^n;\xi_u^n, \xi_p^n)
\\
&\;
-\frac{\tau^{-\alpha}}{\Gamma(2-\alpha)}
\left(\sum\limits_{i=1}^{n-1}(b_{i-1}-b_i)e^{-i\gamma \tau}\int_{\Omega}\eta_u^{n-i}\xi_u^ndx
+b_{n-1}e^{-n\gamma \tau}\int_{\Omega}\eta_u^0\xi_u^ndx\right).
\end{split}
\end{equation}
By the definitions of $\mathcal{P}_{\delta}$ for $u$ and
$\mathcal{P}_{1-\delta}$ for $p$, it is easy to see that
$$
\mathcal{F}_{\Omega}(\eta_u^n,\eta_p^n;\xi_u^n, \xi_p^n))=0.
$$
Based on the stability result (\ref{st}), and notice $\rho>0, \mathcal{P}_{\delta}e_u^{0}=0$, from (\ref{err2}) we can have
\begin{equation}\label{z1}
\begin{split}
\|\xi_u^n\|^2&
+\beta\|\xi_p^n\|^2
=
\sum\limits_{i=1}^{n-1}(b_{i-1}-b_i)e^{-i\gamma \tau}
\int_{\Omega}\xi_u^{n-i}\xi_u^ndx
-\beta\int_{\Omega}\mu^{n}(x)\xi_u^ndx
+\beta\int_{\Omega}\eta_p^n\xi_p^ndx+\mathcal{G},
\end{split}
\end{equation}
where $\beta=\tau^{\alpha}\Gamma(2-\alpha), $ and
\begin{equation*}
\begin{split}
\mathcal{G}=&\;
\rho\beta\int_{\Omega}\eta_u^n\xi_u^ndx
+\int_{\Omega}\eta_u^n\xi_u^ndx
\\
&\;
-\left(
\sum\limits_{i=1}^{n-1}(b_{i-1}-b_i)e^{-i\gamma \tau}
\int_{\Omega}\eta_u^{n-i}\xi_u^nd
+b_{n-1}e^{-n\gamma \tau}\int_{\Omega}\eta_u^n\xi_u^ndx
\right)
\\
=&\;
\rho\beta\int_{\Omega}\eta_u^n\xi_u^ndx
+\tau \sum\limits_{i=0}^{n-1}b_i\int_{\Omega}
\partial_te^{-i\gamma \tau}\eta_u^{n-i}\xi_u^ndx,
\end{split}
\end{equation*}
here the simplified notation is used,
$$
\partial_t\varphi(x,t_{k})=\frac{\varphi(x,t_{k})-\varphi(x,t_{k-1})}{\tau}.
$$
With the help of
\begin{equation*}
\begin{split}
\|\partial_te^{-i\gamma \tau}\eta_u^{n-i}\|
\leq&\;
\|\frac{e^{-T\gamma}}{\tau}\int_{t_{n-i}}^{t_{n-i-1}}
\|\frac{\partial}{\partial t}\eta_u(x,t))dt\|
\leq Ch^{k+1},
\end{split}
\end{equation*}
then we can get
\begin{equation}\label{z11}
\begin{split}
\|\mathcal{G}\|
\leq&\;
\tau \sum\limits_{i=0}^{n-1}b_i
\|\partial_te^{-i\gamma \tau}\eta_u^{n-i}\|\|\xi_u^n\|
+Ch^{k+1}\tau^{\alpha}\|\xi_u^n\|
\\
\leq&\;
Ch^{k+1}\tau \sum\limits_{i=0}^{n-1}b_i\|\xi_u^n\|
+Ch^{k+1}\tau^{\alpha}\|\xi_u^n\|
\\
\leq&\;
Ch^{k+1}\tau n^{1-\alpha}\|\xi_u^n\|
+Ch^{k+1}\tau^{\alpha}\|\xi_u^n\|
\\
\leq&\;
CT^{1-\alpha}h^{k+1}\tau^{\alpha}\|\xi_u^n\|.
\end{split}
\end{equation}
Similarly we have
\begin{equation}\label{z12}
\begin{split}
|-\beta\int_{\Omega}\gamma^{n}(x)\xi_u^ndx|
\leq
C\tau^2\|\xi_u^n\|.
\end{split}
\end{equation}
An application of Young inequality yields
\begin{equation}\label{z13}
\begin{split}
|\beta\int_{\Omega}\eta_p^n\xi_p^ndx|
\leq&\;
\frac{\beta^2}{4\varepsilon}\|\eta_p^n\|^2
+\varepsilon\|\xi_p^n\|^2,
\end{split}
\end{equation}
where $\varepsilon$ is a small constant.

By mean of
$e^{-i\gamma \tau}\leq 1$,
and (\ref{z11}), (\ref{z12}), (\ref{z13}), from  (\ref{z1}) we can obtain
\begin{equation}\label{z1}
\|\xi_u^n\|
\leq
\sum\limits_{i=1}^{n-1}(b_{i-1}-b_i)\|\xi_u^{n-i}\|
+C(\tau^2+h^{k+1}(\tau)^{\alpha}).
\end{equation}
By Lemma \ref{tf3}, we can obtain the following
result immediately
$$
\|\xi_u^n\|\leq C(h^{k+1}+\tau^{2-\alpha}).
$$

Finally, Theorem \ref{errorresult} follows by
the triangle inequality
and Lemma \ref{projection111}.
\end{proof}

\section {\large Numerical examples}
In this section some numerical examples are carried out to
illustrate the
accuracy and capability of the method.
With the help
of successive mesh refinements we have verified that
the scheme is numerically convergent.

\textbf{Example 4.1.}
Considering tempered fractional equation (\ref{question}) in
$\Omega=[0,1]$ with $\rho=0$,
the corresponding forcing term $f(x,t)$ is of the form
\begin{equation}\label{s11}
f(x,t)=\frac{2e^{-\gamma t} t^{2-\alpha}}{\Gamma(3-\alpha)}
\sin(2\pi x)+4 \pi^2t^2e^{-\gamma t}\sin(2\pi x),
\end{equation}
then the exact solution is
$u(x,t)=e^{-\gamma t}t^2\sin(2\pi x)$.
The time step is $\tau=1/M$ and $N$ is the number of the mesh
in space.

First, Spatial accuracy of the scheme (\ref{scheme}) is
tested by taking a sufficiently small
step sizes $\tau=1/1000$ in time. We divide the space into
$N$ elements to form the uniform mesh and
then randomly perturb the coordinates by $10\%$ to construct
the nonuniform mesh.
We list both the $L^2$-norm and $L^\infty$-norm errors,
and the numerical orders of accuracy at time $T=1$ for several
$\alpha$s and $\delta$s,
for both the uniform and nonuniform meshes in Table 1-4.
One can find that the errors attain $(k+1)$-th order of accuracy
for piecewise $P^k$ polynomials.

Secondly, we will test the temporal accuracy of the scheme (\ref{scheme}) with generalized alternating
numerical fluxes (\ref{flux1}).  We take a sufficiently small step
sizes $h=1/200$ so that the space discretization
error is negligible as compared with the time error.
The expected $(2-\alpha)$-th order convergence of
scheme (\ref{scheme}) in time could be seen in Table 5.

\begin{table}[htbp]
\tabcolsep 0pt \caption{Spatial accuracy test on uniform meshes with generalized alternating
numerical fluxes when $\delta=0.3, \gamma=2, M=10^3, T=1$.}
 \vspace*{-24pt}
\begin{center}
\def\temptablewidth{1.0\textwidth}
{\rule{\temptablewidth}{1pt}}
\begin{tabular*}{\temptablewidth}{@{\extracolsep{\fill}}|c|c|ccccccc|}
 $\delta$      &$\alpha$       &  $P^k$  &  $N$ &   $L^2$-error& order & $L^\infty$-error & order&\\
 \cline{1-8}
&&& 5&3.600997655347402E-002 & - &8.470765811325168E-002& -&\\
&&& 10&1.751495701129877E-002& 1.04 &4.247840062543977E-002 &1.00&\\
&&$P^0$ & 20&8.698273697461307E-003& 1.01 &2.125369555415509E-002 & 1.00 &\\
&&& 40&4.341798450317296E-003& 1.00 &1.062863299729018E-002&1.00&\\
 \cline{3-8}
 &&& 5&9.943585411573410E-003& - &2.757075846380118E-002& -&\\
&$\alpha=0.1$&& 10&3.566273983250412E-003& 1.48 &1.080003816776576E-002&1.35&\\
&&$P^1$ &20&1.086919816367809E-003&1.71 &3.420222759716512E-003& 1.66& \\
&&& 40&2.902709222939700E-004& 1.90 &9.193407604523585E-004&1.90&\\
 \cline{3-8}
&& & 5&7.990640423350242E-004 & - &3.417434847568336E-003& -&\\
&&& 10&8.626266947702263E-005& 3.21 &3.694158436061931E-004&3.20&\\
&&$P^2$ & 20&1.044514634698983E-005& 3.04 &4.454609652087860E-005 &3.05& \\
&&& 40&1.296172202016491E-006& 3.01 &5.508580340890919E-006&3.01&\\
 \cline{2-8}
&&& 5&3.594827779072097E-002 & - &8.455751510071857E-002& -&\\
&&& 10&1.750788942725992E-002& 1.04 &4.246067080980019E-002 &1.00&\\
&&$P^0$ & 20&8.697412259699867E-003& 1.01 &2.125151852856288E-002 & 1.00& \\
&&& 40&4.341693080827747E-003& 1.00 &1.062836627483681E-002&1.00&\\
 \cline{3-8}
& && 5&9.921671652323443E-003& - &2.759417252018699E-002& -&\\
&$\alpha=0.6$&& 10&3.563738012224470E-003& 1.48 &1.080391242567882E-002&1.35&\\
&&$P^1$ &20&1.086707146130006E-003&1.71 &3.420472591014495E-003& 1.66& \\
$\delta=0.3$&&& 40&2.902558645272467E-004& 1.90 &9.194266256405403E-004&1.90&\\
 \cline{3-8}
 &&& 5&7.988664990181788E-004& - &3.416510378879402E-003& -&\\
&&& 10&8.626077461730029E-005& 3.21 &3.694001433403462E-004&3.20&\\
&&$P^2$ & 20&1.044545573669924E-005& 3.04 &4.454547869921856E-005 &3.05 &\\
&&& 40&1.298706242251523E-006& 3.01 &5.508452148315928E-006&3.01&\\
 \cline{2-8}
&&& 5&3.592329974624144E-002 & - &8.449616965675322E-002& -&\\
&&& 10&1.750510512702476E-002& 1.04 &4.245362255057922E-002 &1.00&\\
&&$P^0$ & 20&8.697106637946338E-003& 1.01 &2.125073937587364E-002& 1.00 &\\
&&& 40&4.341672465121832E-003& 1.00 &1.062831369241851E-002&1.00&\\
 \cline{3-8}
 &&& 5&9.912664929559693E-003& - &2.760460772303402E-002& -&\\
&$\alpha=0.8$&& 10&3.562659912073956E-003& 1.48 &1.080630043834668E-002&1.35&\\
&&$P^1$ &20&1.086605203698557E-003&1.71 &3.421287359547415E-003& 1.66 &\\
&&& 40&2.902460763410707E-004& 1.90 &9.201811924158254E-004&1.90&\\
 \cline{3-8}
 &&& 5&7.988141577793738E-004& - &3.416151504204671E-003& -&\\
&&& 10& 8.626566626691957E-005& 3.21 &3.693913312302985E-004&3.21&\\
&&$P^2$ & 20&1.046396713324052E-005& 3.04 &4.454319799217823E-005 &3.05 &\\
&&& 40&1.436914611732205E-006& 2.86 &5.619072907782352E-006&2.99&\\
 \cline{1-8}
 \toprule
 \end{tabular*}
\end{center}
 \end{table}

\begin{table}[htbp]
\tabcolsep 0pt \caption{Spatial accuracy test on uniform meshes with generalized alternating
numerical fluxes when $\delta=0.1$,~~  $\gamma=2, M=10^3, T=1$.}
 \vspace*{-24pt}
\begin{center}
\def\temptablewidth{1.0\textwidth}
{\rule{\temptablewidth}{1pt}}
\begin{tabular*}{\temptablewidth}{@{\extracolsep{\fill}}|c|c|ccccccc|}
 $\delta$      &$\alpha$       &  $P^k$  &  $N$ &   $L^2$-error& order & $L^\infty$-error & order&\\
 \cline{1-8}
&&& 5&3.600997655347402E-002 & - &8.470765811325168E-002& -&\\
&&& 10&1.751495701129877E-002& 1.04 &4.247840062543977E-002 &1.00&\\
&&$P^0$ & 20&8.698273697461307E-003& 1.01 &2.125369555415509E-002& 1.00 &\\
&&& 40&4.341798450317296E-003& 1.00 &1.062863299729018E-002&1.00&\\
 \cline{3-8}
&& & 5&9.125860752324633E-003& - &3.386112146946083E-002& -&\\
&$\alpha=0.1$&& 10&2.295048237777114E-003& 1.99 &8.756550512034528E-003&1.95&\\
&&$P^1$ &20&5.745423502445809E-004&2.00 &2.207822490277067E-003& 1.98& \\
&&& 40&1.436832777579926E-004& 2.00 &5.554382129445423E-004&1.99&\\
 \cline{3-8}
 &&& 5&9.050669939275690E-004& - &4.298666679244556E-003& -&\\
&&& 10& 1.151488108051345E-004& 2.97 & 5.375579180586712E-004&3.00&\\
&&$P^2$ & 20&1.445721437410244E-005& 2.99 &6.924567827167210E-005 &2.96 &\\
&&& 40&1.809466716341844E-006& 3.00 &8.720509756160153E-006&2.99&\\
 \cline{2-8}
 &&& 5&3.594827779072097E-002 & - &8.455751510071857E-002& -&\\
&&& 10&1.750788942725992E-002& 1.04 &4.246067080980019E-002 &1.00&\\
&&$P^0$ & 20&8.697412259699867E-003& 1.01 &2.125151852856288E-002& 1.00 &\\
$\delta=0.1$&&& 40&4.341693080827747E-003& 1.00 &1.062836627483681E-002&1.00&\\
\cline{3-8}
 &&& 5&9.121793113875693E-003& - &3.384039722034468E-002& -&\\
&$\alpha=0.6$&& 10&2.294831526335838E-003& 1.99 &8.755484555290433E-003&1.95&\\
&&$P^1$ &20&5.745294690001247E-004&2.00 &2.207828785762922E-003& 1.98& \\
&&& 40&1.436825323167711E-004& 2.00 &5.555100933167800E-004&1.99&\\
 \cline{3-8}
 &&& 5&9.047609877074052E-004& - &4.297245174444950E-003& -&\\
&&& 10& 1.151394966919664E-004& 2.97 &5.375058641304305E-004&3.00&\\
&&$P^2$ & 20&1.445715309857497E-005& 2.99 &6.924406859557769E-005 &2.96 &\\
&&& 40&1.811270951788605E-006& 3.00 &8.720463021409203E-006&2.99&\\
 \cline{2-8}
&&& 5&3.592329974624144E-002& - &8.449616965675322E-002& -&\\
&&& 10&1.750510512702476E-002& 1.04 &4.245362255057922E-002 &1.00&\\
&&$P^0$ & 20&8.697106637946338E-003& 1.01 &2.125073937587364E-002& 1.00 &\\
&&& 40&4.341672465121832E-003& 1.00 &1.062831369241851E-002&1.00&\\
 \cline{3-8}
 &&& 5&9.120195825181981E-003& - &3.383289362730243E-002& -&\\
&$\alpha=0.8$&& 10&2.294751340921533E-003& 1.99 &8.755789596454649E-003&1.95&\\
&&$P^1$ &20&5.745261460924934E-004&2.00 &2.208566384755528E-003& 1.98 &\\
&&& 40&1.436839091653192E-004& 2.00 &5.562652869576801E-004&1.99&\\
 \cline{3-8}
 &&& 5&9.046409059470311E-004& - &4.297126983902264E-003& -&\\
&&& 10&1.151376194022218E-004& 2.97 &5.374880423776833E-004&3.00&\\
&&$P^2$ & 20&1.447012828175748E-005& 2.99 &6.924379906001369E-005 &2.96 &\\
&&& 40&1.912757116440075E-006& 2.99 &8.835769827388040E-006&2.99&\\
\cline{1-8}
 \toprule
 \end{tabular*}
\end{center}
 \end{table}

\begin{table}[htbp]
\tabcolsep 0pt \caption{Spatial accuracy test on nonuniform meshes with generalized alternating
numerical fluxes when $\delta=0.3$,~~  $\gamma=2, M=10^3, T=1$.}
 \vspace*{-24pt}
\begin{center}
\def\temptablewidth{1.0\textwidth}
{\rule{\temptablewidth}{1pt}}
\begin{tabular*}{\temptablewidth}{@{\extracolsep{\fill}}|c|c|ccccccc|}
 $\delta$      &$\alpha$       &  $P^k$  &  $N$ &   $L^2$-error& order & $L^\infty$-error & order&\\
 \cline{1-8}
&&& 5&6.559598407976026E-002 & - &0.141942066666751& -&\\
&&& 10&2.141331033721447E-002& 1.61 &5.734028406382948E-002 &1.30&\\
&&$P^0$ & 20&9.793316970413680E-003& 1.12 &2.815183989101257E-002& 1.02 &\\
&&& 40&4.809485138977877E-003& 1.02 &1.407535025139766E-002&1.00&\\
 \cline{3-8}
&& & 5&9.334237440748634E-003& - &3.145978397121429E-002& -&\\
&$\alpha=0.1$&& 10&3.638112329117496E-003& 1.36 &1.339194646885009E-002&1.23&\\
&&$P^1$ &20&1.106145242636678E-003&1.72 &4.300814221832788E-003& 1.64& \\
&&& 40&2.952217379902808E-004& 1.91 &1.195193749833484E-003&1.85&\\
 \cline{3-8}
 &&& 5&1.010581921645606E-003& - &5.126434823465115E-003& -&\\
&&& 10&1.068201999084790E-004& 3.24 &5.971277580741609E-004&3.10&\\
&&$P^2$ & 20&1.389860346772117E-005& 2.94 &8.309553534868077E-005 &2.85 &\\
&&& 40&1.770704698351656E-006& 2.97 &1.076235396982232E-005&2.95&\\
\cline{2-8}
 &&& 5&6.519959540296516E-002 & - &0.141364976011094& -&\\
&&& 10&2.137597683554834E-002& 1.60 &5.728856602677920E-002 &1.31&\\
&&$P^0$ & 20&9.788419793457237E-003& 1.12 &2.814456788448857E-002& 1.02 &\\
$\delta=0.3$&&& 40&4.808850888477931E-003& 1.01 &1.407438548005000E-002&1.00&\\
\cline{3-8}
 &&& 5&9.324055653453555E-003& - &3.148494472151778E-002& -&\\
&$\alpha=0.6$&& 10&3.637084108871119E-003& 1.36 &1.339218081832141E-002&1.23&\\
&&$P^1$ &20&1.106059577072876E-003&1.72 &4.300878545495546E-003& 1.64& \\
&&& 40&2.952158819276407E-004& 1.91 &1.195178864173474E-003&1.85&\\
 \cline{3-8}
 &&& 5&1.010440003914728E-003& - &5.125652210367676E-003& -&\\
&&& 10&1.068162274307204E-004& 3.23 &5.971411497308199E-004&3.10&\\
&&$P^2$ & 20&1.389837000772890E-005& 2.94 &8.309427296208810E-005 &2.85&\\
&&& 40&1.770454293076160E-006& 2.97 &1.076230665963618E-005&2.95&\\
 \cline{2-8}
 &&& 5&6.458019587494233E-002& - &0.140458685661883& -&\\
&&& 10&2.131839260620296E-002& 1.60 &5.720848587853570E-002&1.31&\\
&&$P^0$ & 20&9.780967395542408E-003& 1.12 &2.813343523512223E-002& 1.01 &\\
&&& 40&4.807925152620669E-003& 1.02 &1.407295053444370E-002&1.00&\\
 \cline{3-8}
 &&& 5&9.308221701165681E-003& - &3.152477270193341E-002& -&\\
&$\alpha=0.8$&& 10&3.635461850529788E-003& 1.36 &1.339294605022746E-002&1.24&\\
&&$P^1$ &20&1.105916081128332E-003&1.72 &4.301517739581207E-003& 1.64 &\\
&&& 40&2.952039890812090E-004& 1.91 &1.195672409974230E-003&1.85&\\
 \cline{3-8}
 &&& 5&1.010240608178436E-003& - &5.124773588667655E-003& -&\\
&&& 10&1.068123075153382E-004& 3.23 &5.971640034306184E-004&3.10&\\
&&$P^2$ & 20&1.390323465425386E-005& 2.94 &8.309265950586455E-005&2.85 &\\
&&& 40&1.809513967454312E-006& 2.98 &1.076227541254043E-005&2.95&\\
\cline{1-8}
 \toprule
 \end{tabular*}
\end{center}
 \end{table}

\begin{table}[htbp]
\tabcolsep 0pt \caption{Spatial accuracy test on nonuniform meshes with generalized alternating
numerical fluxes when $\delta=0.1$,~~  $\gamma=2, M=10^3, T=1$.}
 \vspace*{-24pt}
\begin{center}
\def\temptablewidth{1.0\textwidth}
{\rule{\temptablewidth}{1pt}}
\begin{tabular*}{\temptablewidth}{@{\extracolsep{\fill}}|c|c|ccccccc|}
 $\delta$      &$\alpha$       &  $P^k$  &  $N$ &   $L^2$-error& order & $L^\infty$-error & order&\\
 \cline{1-8}
&&& 5&4.554426032174388E-002 & - &0.115718177138195 & -&\\
&&& 10&1.907614555969082E-002& 1.25 &5.320421957088832E-002 &1.12&\\
&&$P^0$ & 20&9.138240767269679E-003& 1.06 &2.598259449662434E-002& 1.03 &\\
&&& 40&4.518863802087745E-003& 1.01 &1.291377789093143E-002&1.00&\\
 \cline{3-8}
&& & 5&8.849830180737654E-003& - &3.969655576082662E-002& -&\\
&$\alpha=0.1$&& 10&2.670888776568979E-003& 1.73 &1.117674633593155E-002&1.83&\\
&&$P^1$ &20&6.856209305073176E-004&1.96 &2.993783590210047E-003& 1.90& \\
&&& 40&1.726263868565285E-004& 1.99 &7.624154117570892E-004&1.97&\\
 \cline{3-8}
 &&& 5&1.082053891229540E-003& - &5.632531555489165E-003& -&\\
&&& 10&1.177827066011148E-004& 3.20 &7.864352387506149E-004&2.84&\\
&&$P^2$ & 20&1.464825722091804E-005& 3.01 &9.550948635494300E-005 &3.04 &\\
&&& 40&1.829209283100811E-006& 3.00 &1.207411590455412E-005&2.98&\\
\cline{2-8}
&&& 5&4.543056490078846E-002 & - &0.115474889448670& -&\\
&&& 10&1.906295169538061E-002& 1.25 &5.318082033209643E-002 &1.12&\\
&&$P^0$ & 20&9.136635333555037E-003& 1.06 &2.597982073437711E-002& 1.03 &\\
$\delta=0.1$&&& 40&4.518663596834055E-003& 1.01 &1.291343451385024E-002&1.00&\\
\cline{3-8}
 &&& 5&8.846862605510398E-003& - &3.970373991829590E-002& -&\\
&$\alpha=0.6$&& 10&2.670694860720131E-003& 1.73 &1.117614841378017E-002&1.83&\\
&&$P^1$ &20&6.856086070180152E-004&1.96&2.993777705012829E-003& 1.90& \\
&&& 40&1.726256192263683E-004& 1.99 &7.623984725640964E-004&1.97&\\
 \cline{3-8}
 &&& 5&1.081894553798753E-003& - &5.631794818602251E-003& -&\\
&&& 10&1.177789297367418E-004& 3.20 &7.864212743115771E-004&2.84&\\
&&$P^2$ & 20&1.464811122465675E-005& 3.01 &9.550366030672275E-005 &3.04&\\
&&& 40&1.828970584365630E-006& 3.00 &1.207133596128895E-005&2.99&\\
 \cline{2-8}
 &&& 5&4.525430019654393E-002& - &0.115095429944861& -&\\
&&& 10&1.904269469568265E-002& 1.25 &5.314471217556128E-002&1.12&\\
&&$P^0$ & 20&9.134222951010450E-003& 1.06 &2.597562042326319E-002& 1.03 &\\
&&& 40&4.518387505390618E-003& 1.01 &1.291295179596993E-002&1.00&\\
 \cline{3-8}
 &&& 5&8.842271383522345E-003& - &3.971560941693522E-002& -&\\
&$\alpha=0.8$&& 10&2.670392762456530E-003& 1.73 &1.117577463774522E-002&1.83&\\
&&$P^1$ &20&6.855884088997946E-004&1.96 &2.994298251659394E-003& 1.90 &\\
&&& 40&1.726244848834557E-004& 1.99 &7.628868021073432E-004&1.97&\\
 \cline{3-8}
 &&& 5&1.081660399859714E-003& - &5.630991561418289E-003& -&\\
&&& 10&1.177745623909475E-004& 3.20 &7.864030096010258E-004&2.84&\\
&&$P^2$ & 20&1.465280637729603E-005& 3.01 &9.565506409087154E-005&3.04 &\\
&&& 40&1.866808245534035E-006& 3.00 &1.214811102974098E-005&2.99&\\
\cline{1-8}
 \toprule
 \end{tabular*}
\end{center}
 \end{table}

\begin{table}[t]
\tabcolsep 0pt \caption{ Temporal accuracy test using piecewise $P^2$ polynomials for the scheme (\ref{scheme}) with generalized alternating numerical fluxes when $N=100, T=1.$}
 \vspace*{-24pt}
\begin{center}
\def\temptablewidth{1.0\textwidth}
{\rule{\temptablewidth}{1pt}}
\begin{tabular*}{\temptablewidth}{@{\extracolsep{\fill}}c|c|ccccccc|}
&$\delta$  &$\alpha$  &  $\tau$ &   $L^2$-error& order & $L^\infty$-error & order&\\
 \cline{2-8}
&& & 0.04&8.608763604447880E-006& - &1.219684581693636E-005& - &\\
&&& 0.02& 3.086574970641430E-006 & 1.48 &4.409647689024299E-006& 1.47& \\
&&$\alpha=0.5$ &0.01&1.109311333958263E-006& 1.48&1.667659212722938E-006&1.40&  \\
&&& 0.005& 4.122054755449973E-007& 1.43&6.985320445991206E-007& 1.26 &\\
 \cline{3-8}
&$\delta=0.1$&& 0.04&2.391687146347450E-005& - &3.383563665176892E-005& -& \\
&&& 0.02&9.853827361471093E-006 & 1.28&1.395429884359922E-005& 1.28 &\\
&&$\alpha=0.7$ &0.01&4.116933701786636E-006& 1.26&5.856072709420346E-006& 1.26 & \\
&&& 0.005& 1.782609365838843E-006& 1.21&2.587548893207003E-006& 1.18 &\\
 \cline{2-8}
 &&& 0.04&8.608382726939050E-006& - &1.217526939467639E-005& -& \\
&&& 0.02& 3.085511558119693E-006 & 1.48 &4.377333736760303E-006& 1.48& \\
&&$\alpha=0.5$ &0.01&1.106348076087462E-006& 1.48&1.601411114160456E-006&1.45 & \\
&&& 0.005& 4.041621399151297E-007& 1.45&6.654877294648420E-007& 1.27 &\\
 \cline{3-8}
&$\delta=0.3$&& 0.04&2.391673470774867E-005& - &3.382490687975359E-005& - &\\
&&& 0.02&9.853494675524166E-006 & 1.28&1.393619582018557E-005& 1.28& \\
&&$\alpha=0.7$ &0.01&4.116136624036466E-006& 1.26&5.832318129950220E-006& 1.26 & \\
&&& 0.005& 1.780767052630476E-006& 1.21&2.542037339958725E-006& 1.20 &\\
 \cline{2-8}
 \end{tabular*}
\end{center}
 \end{table}

\textbf{Example 4.2.}
Let us continue to consider the problem (\ref{question})
with $\rho=1$ and the exact solution
$$
u(x,t)=e^{-\gamma t}t^2x^2(1-x)^2.
$$
The forcing term on the right-hand side is determined by
the exact solution.

We implement the scheme (\ref{scheme}) in $\Omega=[0,1]$.
Tables 6-9 show the convergence orders for the $L^2$-norm
and $L^\infty$-norm errors at time $T=1$
for several $\alpha$s and $\delta$s
on the uniform mesh and nonuniform mesh, respectively.
The optimal convergence orders in these tables show
that the result in Theorem 3.2 is sharp.

\begin{table}[htbp]
\tabcolsep 0pt \caption{Spatial accuracy test on uniform meshes with generalized alternating
numerical fluxes when $\rho=1, \delta=0.2$,~~  $\gamma=2, M=10^3, T=1$.}
 \vspace*{-24pt}
\begin{center}
\def\temptablewidth{1.0\textwidth}
{\rule{\temptablewidth}{1pt}}
\begin{tabular*}{\temptablewidth}{@{\extracolsep{\fill}}|c|c|ccccccc|}
 $\delta$      &$\alpha$       &  $P^k$  &  $N$ &   $L^2$-error& order & $L^\infty$-error & order&\\
 \cline{1-8}
&&& 5&1.582340814827774E-003 & - &3.587747814566711E-003& -&\\
&&& 10&6.072920572041250E-004& 1.38 &1.361992544606309E-003 &1.39&\\
&&$P^0$ & 20&2.781722995839050E-004& 1.12 &6.593363827866452E-004 & 1.05 &\\
&&& 40&1.358688098753038E-004& 1.03 &3.262244021583708E-004&1.01&\\
 \cline{3-8}
 &&& 5&2.976903263874860E-004& - &8.888076108539667E-004& -&\\
&$\alpha=0.3$&& 10&9.380682880588326E-005& 1.67 &3.747599073017630E-004&1.25&\\
&&$P^1$ &20&2.570892164256014E-005&1.87 &1.197734627369952E-004& 1.65& \\
&&& 40&6.616980217734994E-006& 1.96 &3.372607623089679E-005&1.83&\\
 \cline{3-8}
&& & 5&3.744225938986966E-005 & - &2.066750403204734E-004& -&\\
&&& 10&4.242629022151605E-006& 3.14 &2.905747077063567E-005&2.83&\\
&&$P^2$ & 20&5.108240307324746E-007& 3.05 &3.733062412131737E-006&2.96& \\
&&& 40&1.404276778729993E-007& 1.86 &4.548930764244681E-007&3.03&\\
 \cline{2-8}
&&& 5&1.576014035107718E-003& - &3.577323483172182E-003& -&\\
&&& 10&6.065883866688050E-004& 1.38 &1.361430273966585E-003&1.39&\\
&&$P^0$ & 20&2.780870829316098E-004& 1.12 &6.589782886991372E-004 & 1.05& \\
&&& 40&1.358586702912018E-004& 1.03 &3.262537733518698E-004&1.01&\\
 \cline{3-8}
& && 5&2.975626899907474E-004& - &8.888043631958508E-004& -&\\
&$\alpha=0.5$&& 10&9.379482889619085E-005& 1.68 &3.746166818088215E-004&1.25&\\
&&$P^1$ &20&2.570772139874646E-005&1.87 &1.196172349767182E-004& 1.65& \\
$\delta=0.2$&&& 40&6.615791862540715E-006& 1.96 &3.356898366286719E-005&1.83&\\
 \cline{3-8}
 &&& 5&3.743605733712328E-005& - &2.068108840347286E-004& -&\\
&&& 10&4.240769769090305E-006& 3.14 &2.921541320911186E-005&2.82&\\
&&$P^2$ & 20&4.961774292224783E-007& 3.09 &3.890639954380062E-006 &2.90 &\\
&&& 40&7.051223319124203E-008& 2.81 &5.312153408062071E-007&2.87&\\
 \cline{2-8}
&&& 5&1.569573150535817E-003 & - &3.566532563472458E-003& -&\\
&&& 10&6.058781928822529E-004& 1.37 &1.361570566765813E-003 &1.39&\\
&&$P^0$ & 20&2.780022433132069E-004& 1.12 &6.584717897009694E-004& 1.04 &\\
&&& 40&1.358494586510882E-004& 1.03 &3.265674113640243E-004&1.01&\\
 \cline{3-8}
 &&& 5&2.974342179950080E-004& - &8.886664664762281E-004& -&\\
&$\alpha=0.7$&& 10&9.378333026032094E-005& 1.67 &3.743331911596685E-004&1.25&\\
&&$P^1$ &20&2.570912149705274E-005&1.87 &1.193198577654686E-004& 1.65 &\\
&&& 40&6.624911739109129E-006& 1.96 &3.327059953717988E-005&1.84&\\
 \cline{3-8}
&&& 5&3.743175301428064E-005& - &2.070928579500273E-004& -&\\
&&& 10&4.255010460194694E-006& 3.13 &2.951598305808386E-005&2.81&\\
&&$P^2$ & 20&6.067267031675130E-007& 2.81 &4.189879914070192E-006 &2.82 &\\
&&& 40& 8.591886309964037E-008&2.82 &6.269178180106001E-007&2.74&\\
 \cline{1-8}
 \toprule
 \end{tabular*}
\end{center}
 \end{table}

\begin{table}[htbp]
\tabcolsep 0pt \caption{Spatial accuracy test on uniform meshes with generalized alternating
numerical fluxes when $\rho=1, \delta=0.6$,~~  $\gamma=2, M=10^3, T=1$.}
 \vspace*{-24pt}
\begin{center}
\def\temptablewidth{1.0\textwidth}
{\rule{\temptablewidth}{1pt}}
\begin{tabular*}{\temptablewidth}{@{\extracolsep{\fill}}|c|c|ccccccc|}
 $\delta$      &$\alpha$       &  $P^k$  &  $N$ &   $L^2$-error& order & $L^\infty$-error & order&\\
 \cline{1-8}
&&& 5&2.263070303841627E-003 & - &4.577360163679048E-003& -&\\
&&& 10&6.707774222704496E-004& 1.75 &1.436684775454216E-003 &1.67&\\
&&$P^0$ & 20&2.853444715298516E-004& 1.23 &6.606848412090876E-004 & 1.12 &\\
&&& 40&1.367407768126421E-004& 1.06 &3.273628226099778E-004&1.01&\\
 \cline{3-8}
&&& 5&2.974509312553489E-004& - &7.150301084008492E-004& -&\\
&$\alpha=0.3$&& 10&1.271245501318707E-004& 1.22 &3.981226597552988E-004&0.84&\\
&&$P^1$ &20&5.095203810569959E-005&1.32&1.756826753133757E-004& 1.18& \\
&&& 40&1.628155726102348E-005& 1.65 &6.113921628532708E-005&1.52&\\
 \cline{3-8}
 && & 5&4.741117314917393E-005 & - &2.322474249148421E-004& -&\\
&&& 10&5.251072284968953E-006& 3.17 &3.482052702302946E-005&2.73&\\
&&$P^2$ & 20&5.487061550738519E-007& 3.26 &4.391297775017012E-006&2.99& \\
&&& 40&1.398475833157788E-007& 1.97 &5.076183265380021E-007&3.11&\\
 \cline{2-8}
&&& 5&2.244975493970608E-003& - &4.553595493445027E-003& -&\\
&&& 10&6.694913740987843E-004& 1.75 &1.433071024540420E-003&1.66&\\
&&$P^0$ & 20&2.851933816998400E-004& 1.23 &6.606969556207020E-004& 1.12& \\
&&& 40&1.367226188180968E-004& 1.06 &3.275293007699463E-004&1.01&\\
 \cline{3-8}
& && 5&2.970077416293948E-004& - &7.150618409357242E-004& -&\\
&$\alpha=0.5$&& 10&1.270399125198733E-004& 1.23 &3.980425972542222E-004&0.85&\\
&&$P^1$ &20&5.093888725451710E-005&1.32 &1.755280846267336E-004& 1.18& \\
$\delta=0.6$&&& 40&1.627971177773421E-005& 1.65 &6.098150580908531E-005&1.53&\\
 \cline{3-8}
 &&& 5&4.738891552206912E-005& - &2.323268736959109E-004& -&\\
&&& 10&5.248488428614583E-006& 3.17 &3.497357272207592E-005&2.73&\\
&&$P^2$ & 20&5.350808403013323E-007& 3.29 &4.548739501698698E-006&2.94 &\\
&&& 40&6.934968715889372E-008& 2.95 &6.122329160500268E-007&2.89&\\
 \cline{2-8}
 &&& 5&2.226364987023399E-003 & - &4.528968446397434E-003& -&\\
&&& 10&6.681909572138881E-004& 1.73 &1.429260382731581E-003 &1.66&\\
&&$P^0$ & 20&2.850422396157703E-004& 1.23 &6.608559171080138E-004& 1.11 &\\
&&& 40&1.367054729489756E-004& 1.06 &3.278432462824107E-004&1.01&\\
 \cline{3-8}
 &&& 5&2.965586452297552E-004& - &7.149566598926194E-004& -&\\
&$\alpha=0.7$&& 10&1.269543669818585E-004& 1.22 &3.978234263835360E-004&0.85&\\
&&$P^1$ &20&5.092673020318355E-005&1.32 &1.752325021227329E-004& 1.18 &\\
&&& 40&1.628197579228578E-005& 1.65 &6.068249513503552E-005&1.53&\\
 \cline{3-8}
&&& 5&4.736816730385127E-005& - &2.325522068324162E-004& -&\\
&&& 10&5.258924978488712E-006& 3.17 &3.526923646273578E-005&2.72&\\
&&$P^2$ & 20&6.389205641380220E-007& 3.04 &4.847842978104736E-006 &2.86 &\\
&&& 40&  8.501389966985875E-008&2.91 & 6.585375965448209E-007&2.88&\\
 \cline{1-8}
 \toprule
 \end{tabular*}
\end{center}
 \end{table}

\begin{table}[htbp]
\tabcolsep 0pt \caption{Spatial accuracy test on nonuniform meshes with generalized alternating
numerical fluxes when $\rho=1, \delta=0.2$,~~  $\gamma=2, M=10^3, T=1$.}
 \vspace*{-24pt}
\begin{center}
\def\temptablewidth{1.0\textwidth}
{\rule{\temptablewidth}{1pt}}
\begin{tabular*}{\temptablewidth}{@{\extracolsep{\fill}}|c|c|ccccccc|}
 $\delta$      &$\alpha$       &  $P^k$  &  $N$ &   $L^2$-error& order & $L^\infty$-error & order&\\
 \cline{1-8}
&&& 5&1.517845954000122E-003& - &3.511606734885603E-003& -&\\
&&& 10&6.227976716020462E-004& 1.28 &1.670434180879555E-003 &1.07&\\
&&$P^0$ & 20&2.917321657674542E-004& 1.09 &8.185201216153321E-004 & 1.02 &\\
&&& 40&1.431224670285340E-004& 1.02 &4.068120471581059E-004&1.01&\\
 \cline{3-8}
&&& 5&3.242733095936642E-004& - &1.031247407943849E-003& -&\\
&$\alpha=0.3$&& 10&9.642143896521544E-005& 1.75 &4.664783993685712E-004&1.14&\\
&&$P^1$ &20&2.638283335064153E-005&1.87&1.497589900759262E-004& 1.64& \\
&&& 40&6.784781440645293E-006& 1.96 &4.208948155767745E-005&1.83&\\
 \cline{3-8}
 && & 5&4.371291470422284E-005& - &2.205191157852509E-004& -&\\
&&& 10&4.737332465461207E-006& 3.21 &5.279340967235548E-005&2.06&\\
&&$P^2$ & 20&5.836666419066179E-007& 3.02 &6.846287374035127E-006&2.95& \\
&&& 40&1.451353937788389E-007& 2.01 &7.721987122251807E-007&3.14&\\
 \cline{2-8}
&&& 5&1.511832756597273E-003& - &3.502121979914235E-003& -&\\
&&& 10&6.220893257475209E-004& 1.28 &1.670078148163840E-003&1.06&\\
&&$P^0$ & 20&2.916417957441209E-004& 1.09 &8.186795059360371E-004& 1.03& \\
&&& 40&1.431115461417255E-004& 1.03 &4.069902921658920E-004&1.01&\\
 \cline{3-8}
& && 5&3.241445818441831E-004& - &1.031010873877333E-003& -&\\
&$\alpha=0.5$&& 10&9.640915416227526E-005& 1.75&4.663237471849166E-004&1.14&\\
&&$P^1$ &20&2.638161337326710E-005&1.87 &1.496019341579287E-004& 1.64& \\
$\delta=0.2$&&& 40&6.783621296347129E-006& 1.96 &4.193227445160367E-005&1.83&\\
 \cline{3-8}
 &&& 5&4.370398829079597E-005& - &2.206503138086216E-004& -&\\
&&& 10&4.735605756347349E-006& 3.21 &5.295075267031514E-005&2.06&\\
&&$P^2$ & 20&5.708884776768829E-007& 3.05 &7.003847236448262E-006&2.92 &\\
&&& 40&7.947531188294052E-008& 2.84 &9.294866692183973E-007&2.91&\\
 \cline{2-8}
 &&& 5&1.505710253623656E-003 & - &3.492298886676585E-003& -&\\
&&& 10&6.213745175059577E-004& 1.28 &1.669866295206499E-003 &1.06&\\
&&$P^0$ & 20&2.915517673954438E-004& 1.09 &8.189868083813912E-004& 1.03 &\\
&&& 40&1.431015120142753E-004& 1.03 &4.073153866903747E-004&1.01&\\
 \cline{3-8}
 &&& 5&3.240154726586564E-004& - &1.030637501521249E-003& -&\\
&$\alpha=0.7$&& 10&9.639741117415519E-005& 1.75 &4.660275194837704E-004&1.15&\\
&&$P^1$ &20&2.638293775032076E-005&1.87 &1.493034254027140E-004& 1.64 &\\
&&& 40&6.792517114587242E-006& 1.96 &4.163370092550497E-005&1.84&\\
 \cline{3-8}
&&& 5&4.369666230527429E-005& - &2.209276622101592E-004& -&\\
&&& 10&4.748293818147065E-006& 3.20 &5.325073483141655E-005&2.05&\\
&&$P^2$ & 20&6.692039201471664E-007& 2.83 &7.303070295557957E-006 &2.87 &\\
&&& 40& 8.581080963134642E-008&2.96 & 9.228188670033924E-007&2.98&\\
 \cline{1-8}
 \toprule
 \end{tabular*}
\end{center}
 \end{table}

\begin{table}[htbp]
\tabcolsep 0pt \caption{Spatial accuracy test on nonuniform meshes with generalized alternating
numerical fluxes when $\rho=1, \delta=0.6$,~~  $\gamma=2, M=10^3, T=1$.}
 \vspace*{-24pt}
\begin{center}
\def\temptablewidth{1.0\textwidth}
{\rule{\temptablewidth}{1pt}}
\begin{tabular*}{\temptablewidth}{@{\extracolsep{\fill}}|c|c|ccccccc|}
 $\delta$      &$\alpha$       &  $P^k$  &  $N$ &   $L^2$-error& order & $L^\infty$-error & order&\\
 \cline{1-8}
&&& 5&2.218435560209597E-003 & - &4.988165866475595E-003& -&\\
&&& 10&7.026515160120938E-004& 1.66 &1.615537336821858E-003 &1.63&\\
&&$P^0$ & 20&3.238249412680320E-004& 1.12 &9.344015225186842E-004 & 0.79 &\\
&&& 40&1.711992658914959E-004& 0.92 &5.144905069215206E-004&0.86&\\
 \cline{3-8}
&&& 5&3.213669677505128E-004& - &6.232490424656974E-004& -&\\
&$\alpha=0.3$&& 10&1.281866059118811E-004& 1.33 &3.003107660959630E-004&1.05&\\
&&$P^1$ &20&5.143385599212686E-005&1.32&1.603700097095964E-004& 0.91& \\
&&& 40&1.645991800902195E-005& 1.64 &5.811184200627087E-005&1.46&\\
 \cline{3-8}
 && & 5&5.667492210358004E-005 & - &3.505324110815492E-004& -&\\
&&& 10&5.578235035051399E-006& 3.17 &4.414888437790894E-005&2.99&\\
&&$P^2$ & 20&7.175666478549729E-007& 3.26 &5.738400869813071E-006&2.94& \\
&&& 40&1.571281022897026E-007& 2.19 &8.425543243264881E-007&2.77&\\
 \cline{2-8}
&&& 5&2.200335979338480E-003& - &4.956964974689847E-003& -&\\
&&& 10&7.011795560405282E-004& 1.64 &1.611200638415444E-003&1.62&\\
&&$P^0$ & 20&3.235547410371371E-004& 1.12 &9.339361814586587E-004&0.79& \\
&&& 40&1.711432886897817E-004& 0.92 &5.145417733177721E-004&0.86&\\
 \cline{3-8}
& && 5&3.208864422126183E-004& - &6.227158645046031E-004& -&\\
&$\alpha=0.5$&& 10&1.280994719241301E-004& 1.32 &3.001105667075760E-004&1.05&\\
&&$P^1$ &20&5.142050916436299E-005&1.32 &1.602085136691192E-004&0.91& \\
$\delta=0.6$&&& 40&1.645805875671424E-005& 1.64 &5.795371024010622E-005&1.47&\\
 \cline{3-8}
 &&& 5&5.664565316990849E-005& - &3.502258146088582E-004& -&\\
&&& 10&5.576080115885499E-006& 3.34 &4.398437752902102E-005&2.99&\\
&&$P^2$ & 20&7.071821450175616E-007& 2.98 &5.873580067984358E-006&2.90 &\\
&&& 40&9.970460744355951E-008& 2.83 &8.509068596339163E-007&2.79&\\
 \cline{2-8}
&&& 5&2.181711723758149E-003 & - &4.924677437389718E-003& -&\\
&&& 10&6.996904701082016E-004& 1.63 &1.606649219131593E-003 &1.62&\\
&&$P^0$ & 20&3.232837185701581E-004& 1.11 &9.336145912944857E-004&0.78 &\\
&&& 40&1.710880587147894E-004& 0.92 &5.147405087257530E-004&0.86&\\
 \cline{3-8}
 &&& 5&3.203996862436910E-004& - &6.223220958377807E-004& -&\\
&$\alpha=0.7$&& 10&1.280113679639955E-004& 1.32 &2.997662559044798E-004&1.05&\\
&&$P^1$ &20&5.140813415494180E-005&1.32 &1.599050123649873E-004& 0.91 &\\
&&& 40&1.646026145542672E-005& 1.64 &5.765402843328131E-005&1.47 &\\
 \cline{3-8}
&&& 5&5.661762164027293E-005& - &3.497734484754480E-004& -&\\
&&& 10&5.586181252970688E-006& 3.34 &4.367727087437284E-005&3.00&\\
&&$P^2$ & 20&7.886608314227983E-007& 2.82 &6.172794085018479E-006 &2.82 &\\
&&& 40&  9.631333432841155E-008&3.03 & 1.149606829718052E-006&2.42&\\
 \cline{1-8}
 \toprule
 \end{tabular*}
\end{center}
 \end{table}

 \textbf{Example 4.3.}
In this example we present some numerical solutions
using piecewise $P^2$ polynomials
for the following homogeneous model:
\begin{equation}\label{e2}
_0^CD_t^{\alpha,\gamma}u(x,t)=\frac{\partial^{2}u(x,t)}{\partial x^{2}},
\end{equation}
and the initial conditions is taken as follows:
$$
u(x,0)=e^{-5(x-3)^2}.
$$

Figures 1 and 2 shows the evolution of the solution behavior for different values of $\alpha$ and $\lambda$ with $\delta=0.2$. The numerical results show that the scheme (\ref{scheme}) with the generalized alternating
numerical fluxes (\ref{flux1}) is very effective to handle such problems numerically.

\begin{figure}[htpb]
\centering
\includegraphics[width=0.8\textwidth]{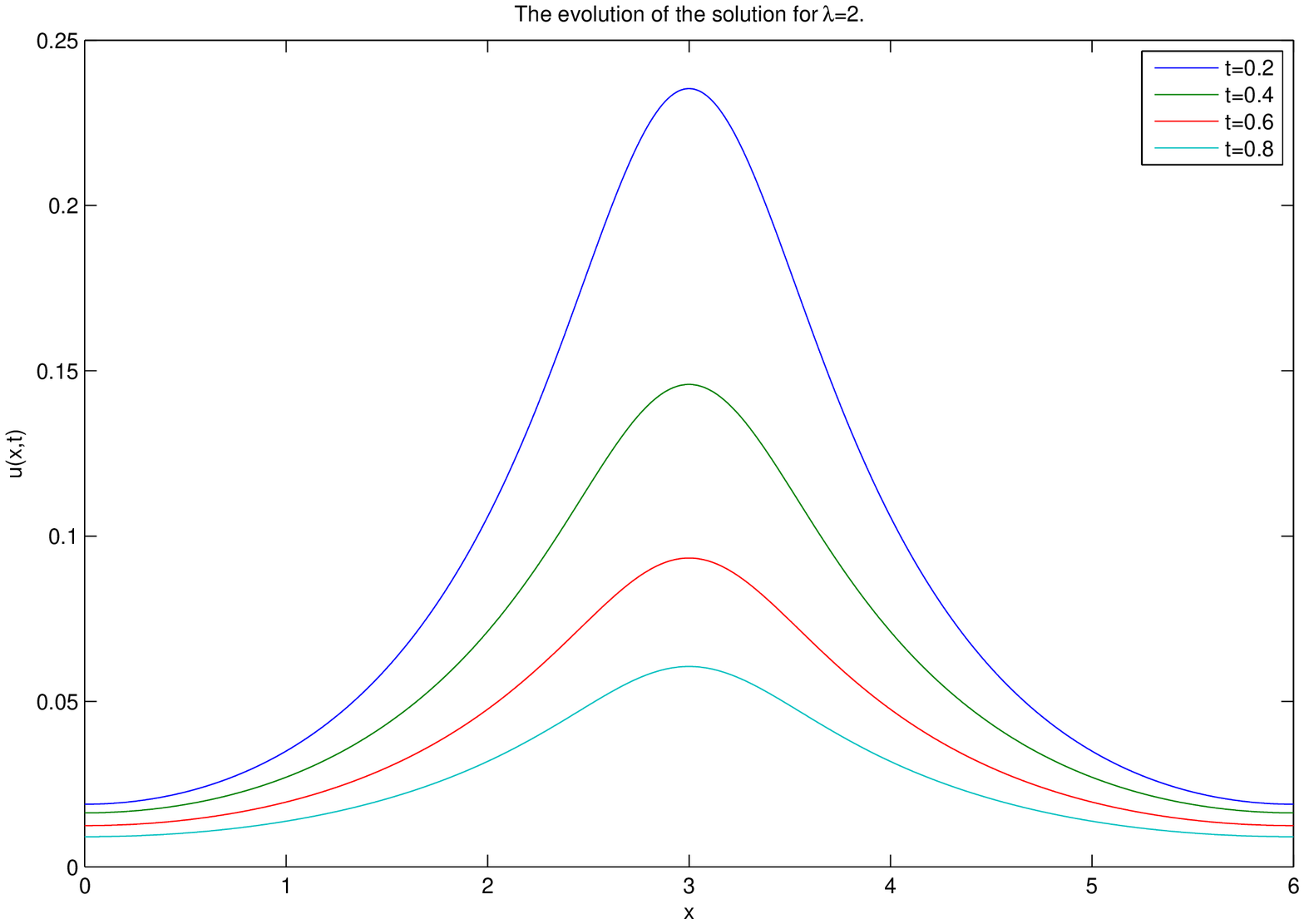}
\includegraphics[width=0.8\textwidth]{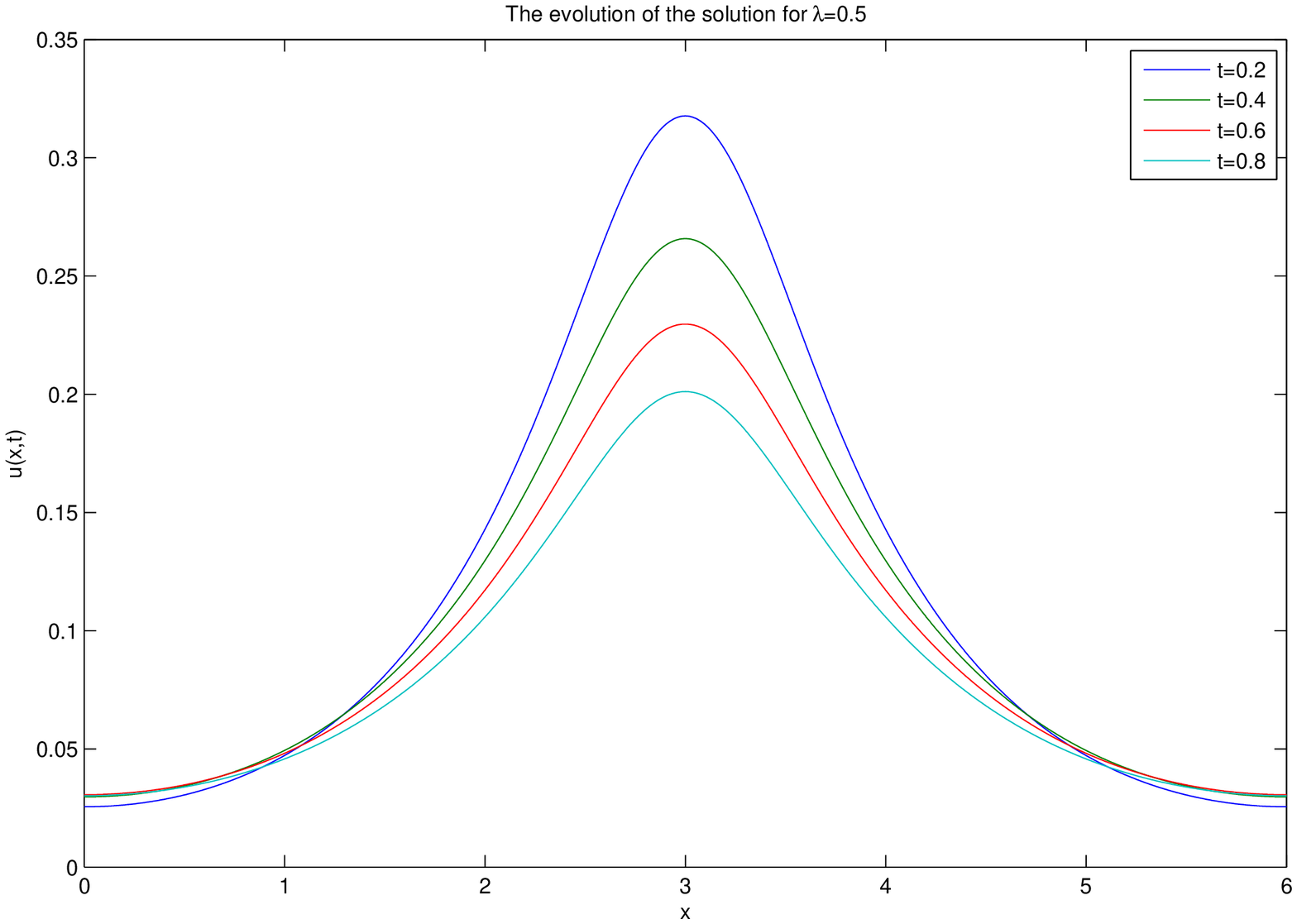}
\caption{The evolution of the solution for $\alpha=0.3$. $\tau=0.001, h=0.01, k=2.$}\label{figure2}
\end{figure}

\begin{figure}[htpb]
\centering
\includegraphics[width=0.8\textwidth]{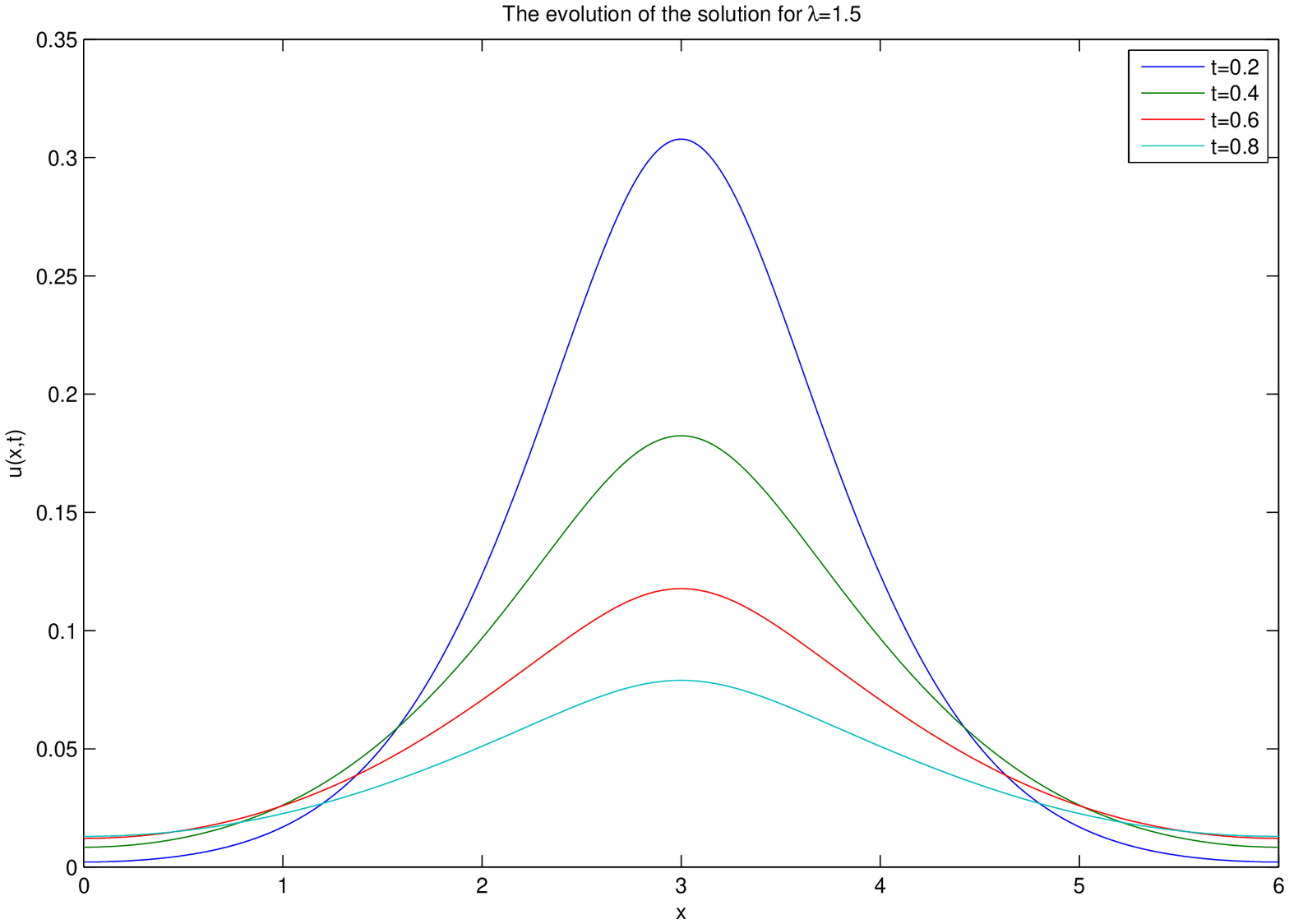}
\includegraphics[width=0.8\textwidth]{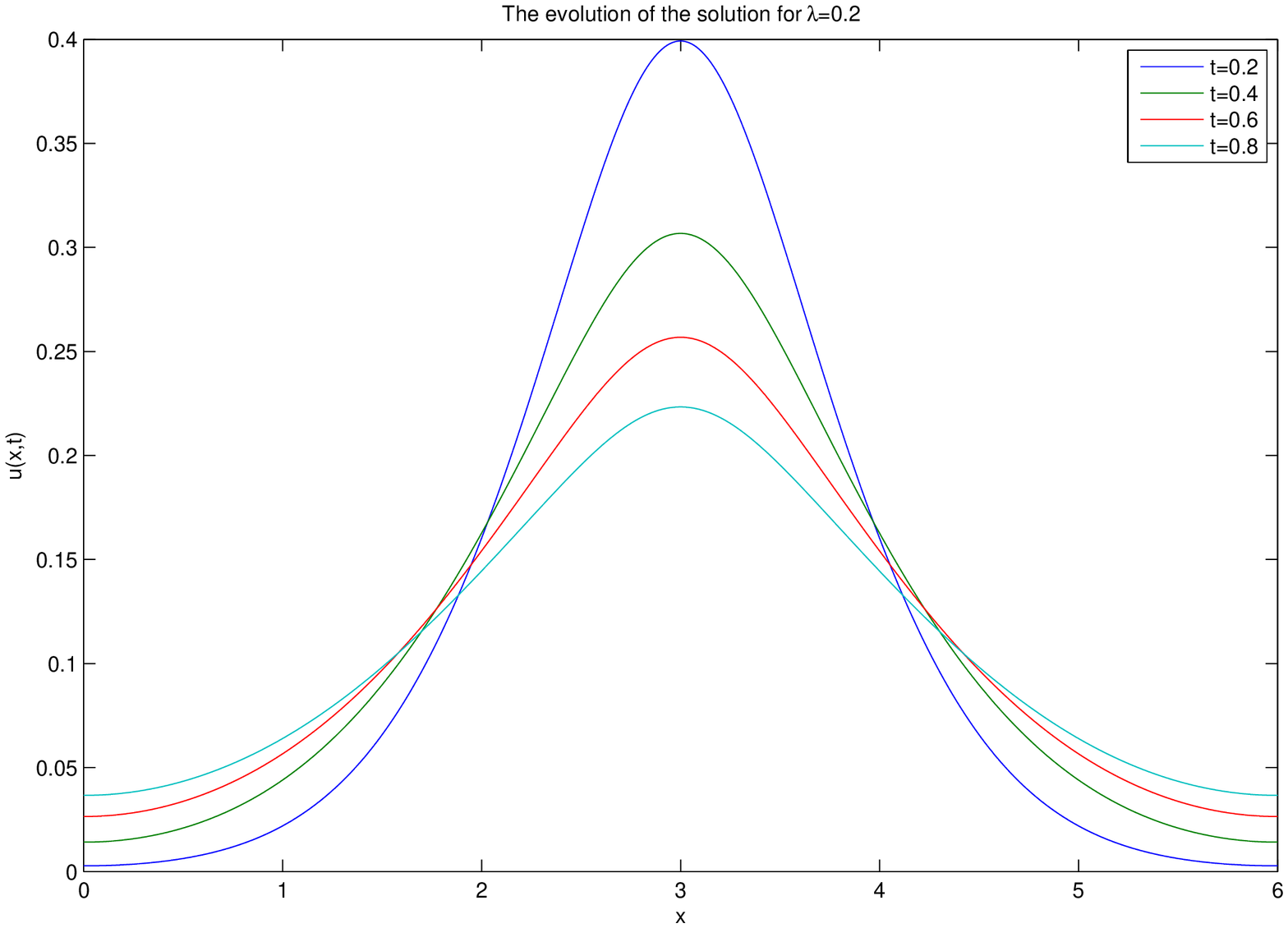}
\caption{The evolution of the solution for $\alpha=0.8$. $\tau=0.001, h=0.01, k=2.$}\label{figure2}
\end{figure}

\section{\large Conclusion}

In this paper we have proposed and analyzed an implicit fully discrete LDG method with the convergent
orders $O(h^{k+1}+(\tau)^{2-\alpha})$ for the tempered fractional diffusion equation.  Numerical examples are provided to confirm the order of convergence and show the effectiveness of the proposed scheme (\ref{scheme}).
It is worth to mention that the scheme can be extended to solve the two or higher dimensional case easily, and the theoretical results are also valid. However, the computation work will be huge. In future we would like to study this problem and try to design a effective scheme for the two dimensional case.

\section*{Acknowledgement}
Supported by the Fundamental Research Funds for the Henan Provincial Colleges and Universities in Henan University of Technology(2018RCJH10), the Training Plan of Young Backbone Teachers in Henan University of Technology(21420049), the Training Plan of Young Backbone Teachers in Colleges and Universities of Henan Province (2019GGJS094), Foundation of Henan Educational Committee(19A110005) and the National Natural Science Foundation of China (11461072).

\end{document}